\documentclass[12pt,a4paper]{amsart}
\usepackage{amsfonts}
\usepackage[top=35mm, bottom=35mm, left=30mm, right=30mm]{geometry}
\usepackage[colorlinks=true,citecolor=blue]{hyperref}
\usepackage{mathptmx}
\usepackage{eucal}
\usepackage{graphicx}
\usepackage{mathrsfs}
\usepackage{amssymb}
\usepackage{amsmath}
\usepackage{amsthm}
\usepackage{tikz,float}

\usepackage{xcolor}
\newtheorem{thm}{Theorem}[section]
\newtheorem{cor}[thm]{Corollary}
\newtheorem{lem}[thm]{Lemma}
\newtheorem{prop}[thm]{Proposition}

\theoremstyle{definition}
\newtheorem{defn}[thm]{Definition}
\newtheorem{exam}[thm]{Example}
\newtheorem{rem}[thm]{Remark}
\numberwithin{equation}{section}
\newtheorem{clm}[thm]{Claim}

\newcommand{\N}{\mathbb{N}}

\DeclareMathOperator{\diam}{diam}

\begin{document}
\title[Equi-invariability, bounded invariance complexity and L-stability]{Equi-invariability, bounded invariance complexity and L-stability for control systems}
\author{Xing-fu Zhong}
\address{School of Mathematics and Statistics, Guangdong University of Foreign Studies\\
Guangzhou, 510006 P. R. China}
\email{xfzhong@gdufs.edu.cn}

\author{Zhi-jing Chen}
\address{School of Mathematics and Systems Science,
	Guangdong Polytechnic Normal University,
	Guangzhou 510665, P.R. China}
\email{chzhjing@mail2.sysu.edu.cn}

\author{Yu Huang}
\address{School of Mathematics and Computational Science, Sun Yat-sen University,
GuangZhou 510275, P. R. China}
\email{stshyu@mail.sysu.edu.cn}



\begin{abstract}
In the paper we introduce the notions of bounded invariance complexity, bounded invariance complexity in the mean and mean L-stability for control systems. Then we characterize these notions by introducing six types of equi-invariability. As by product, two new dichotomy theorems for control system on control sets are established.
\end{abstract}


\maketitle

\section{Introduction}
In this paper, we mainly consider a discrete-time control system on a metric space $X$ of the following form
\begin{equation}\label{eq:control-system-modle}
x_{n+1}=F(x_n,u_n):=F_{u_n}(x_n),~n\in\mathbb{N}_0=\{0,1,\ldots\},
\end{equation}
where $F$ is a map from  $X\times U$ to $X$, $U$ is a compact set, and $F_u(\cdot)\equiv F(\cdot, u)$ is continuous for every $u\in U$. Given a control sequence $\omega=(\omega_0,\omega_1,\ldots)$ in $U$, the solution of~\eqref{eq:control-system-modle} can be written as
\[\phi(k,x,\omega)=F_{\omega_{k-1}}\circ\cdots\circ F_{\omega_0}(x).\]
For convenience, we denote system~\eqref{eq:control-system-modle} by $\Sigma=(\mathbb{N}_0,X,U,\mathscr{U},\phi)$, where $\mathscr{U}=U^{\mathbb{N}_0}$. Furthermore, we assume that $\phi:\mathbb{N}_0\times X\times\mathscr{U}\to X$ is continuous.

Invariance entropy introduced by Colonius and Kawan~\cite{Colonius-Kawan2009} as well as topological feedback entropy introduced by Nair et al.~\cite{Nair2004Topological} characterizes the minimal data rate for making a subset of the state space invariant. It is a very useful invariant to describe the exponential growth rate of the minimal number of different control functions sufficient for orbits to stay in a given set when starting in a subset of this set. For controlled invariant sets with zero invariance entropy, it is useful to consider the invariance complexity function first studied by Wang, Huang and Chen~\cite{Wang2019dichotomy}, which is an analogue in topological dynamical systems (see~\cite{Huang2018Bounded-complexity} and the references therein). We refer the readers to~\cite{Colonius2018Invariance,Colonius2018Metric,Colonius2018PressureB,
Colonius2013Invariance,Colonius-Kawan2009,Colonius2013A,
Colonius2019Bounds,Colonius2019Controllability,Colonius2018PressureA,Huang2018Caratheodory,
Kawan2009A,Kawan2018Invariance,Silva2015Invariance,Silva2013,Silva2014Outer,Wang2019MeasureInvariance,
Zhong2018Invariance,Zhong2020Variational} for more details about invariance entropy.

In 1993, Colonius and Kliemann~\cite{Colonius1993Some} introduced a notion of control set and obtained a beautiful result that control sets of a given control system coincide with maximal topologically mixing (transitive) sets of the control flow induced by the control system under some assumptions. We refer the readers to~\cite{Colonius1993Some,Colonius2000The} for more connections between control properties for control systems and basic notions for dynamical systems. Recently, the authors in~\cite{Wang2019dichotomy} introduced notion of equi-invariability and showed that an equi-invariant compact set has bounded invariance complexity and the converse is not true in general. In particular, they established a dichotomy theorem that a control set with dense interior is either equi-invariant or unstable.

In this paper, we introduce six types of equi-invariability, which are the analogy to equi-continuity, equi-continuity in the mean, and mean equi-continuity in topological dynamical systems (see~\cite{Huang2018Bounded-complexity,Li2015Mean,Qiu2018A}). Then we discuss the relationship with each other. In particular, we use the equi-invariability to characterize the bounded invariance complexity, the bounded invariance complexity in the mean and the mean L-stability for control systems. As by product, we obtain two new dichotomy theorems for a control set with dense interior.

This paper is organized as follows.
 In Section~\ref{sec:Equi-invariability},
 we first introduce six types of equi-invariability and discuss the relationship with each other. Then we characterize bounded invariance complexity, the bounded invariance complexity in the mean and the mean L-stability by equi-invariability. In Section~\ref{sec:Dichotomy-theorem}, we obtain two new  dichotomy theorems for control systems on control sets. All counter-examples are
 given in Appendix.

\section{Finite equi-invariability and invariance complexity}\label{sec:Equi-invariability}
Consider a control system $\Sigma=(\mathbb{N}_0,X,U,\mathscr{U},\phi)$, where $X$ is a metric space with a metric $d$. Recall that a subset
$Q$ of $X$ is said to be \emph{controlled invariant} if for any $x\in Q$, there exists a control $\omega_x\in\mathscr{U}$ such that
$\phi(\mathbb{N}_0,x,\omega_x)\subset Q$. Our task is to keep a controlled invariant set $Q$ invariant. It is almost impossible to realize this
task in practices if the choice of the control $\omega_x$ is sensitive to $x\in Q$ because of the error caused in implementation. On the other hand, we can control such a point $x\in Q$ if the associated control $\omega_x$ keeps not only the orbit of $x$ but also the orbits starting from some neighborhood of $x$ in the nearby of $Q$. Such a point is called a {\it equi-invariant point} of $Q$ in~\cite{Wang2019dichotomy}.
From the view point of control theory, the equi-invariant point $x$ of $Q$ means that $x$ can be stabilized robustly to
any neighborhood of $Q$.

In this section we will introduce six types of equi-invariability and their related to bound invariance complexity for a control system.
\begin{defn}
Let $\Sigma=(\mathbb{N}_0,X,U,\mathscr{U},\phi)$ be a system, $Q\subset X$ be a nonempty set and $x\in Q$.
\begin{itemize}
\item[(1)] $x$ is called a \emph{finitely equi-invariant point} of $Q$, write $x\in \mbox{FEI}(Q)$, if for every $\varepsilon>0$, there exist $\delta>0$ and a finite set $F\subset\mathscr{U}$ such that for every $y\in B(x,\delta)\cap Q$ there exists $\omega\in F$ with
    \begin{equation}\label{eq: equi-invaraint}
      \phi(\mathbb{N}_0,y,\omega)\subset B_\varepsilon(Q).
    \end{equation}
   $Q$ is called a \emph{finitely equi-invariant set} if $\mbox{FEI}(Q)=Q$.
\item[(2)] $x$ is called a \emph{finitely equi-invariant point in the mean} of $Q$, write $x\in \mbox{FEIM}(Q)$, if $x$ satisfies item~(1) where
   the equation~(\ref{eq: equi-invaraint}) is replaced by
    \begin{equation}\label{eq: equi invariant in mean}
     \frac1{n}\sum_{i=0}^{n-1}d[\phi(i,y,\omega),Q]<\varepsilon,~\forall~n\in\mathbb{N}.
    \end{equation}
    $Q$ is called a \emph{finitely equi-invariant set in the mean} if $\mbox{FEIM}(Q)=Q$.
\item[(3)] $x$ is called a \emph{finitely mean equi-invariant point} of $Q$, write $x\in \mbox{FMEI}(Q)$, if $x$ satisfies item~(1) where the equation~(\ref{eq: equi-invaraint}) is replaced by
    \begin{equation}\label{eq: mean equi invariant}
     \limsup_{n\to\infty}\frac1{n}\sum_{i=0}^{n-1}d(\phi(i,y,\omega),Q)<\varepsilon.
    \end{equation}
    $Q$ is called a \emph{finitely mean equi-invariant set} if $\mbox{FMEI}(Q)=Q$.
\end{itemize}
 Furthermore, $x$ is said to be an \emph{equi-invariant point } of $Q$, an \emph{equi-invariant point in the mean} of $Q$ and
 a \emph{mean equi-invariant point} of $Q$ if $x$ satisfies (1)-(3), respectively, with the set $F$ being singleton. We write
 $x\in \mbox{EI}(Q)$, $x\in \mbox{EIM}(Q)$ and $x\in \mbox{MEI}(Q)$, respectively.
\end{defn}

\begin{rem}
\begin{itemize}
 \item[(i)] The concept of equi-invariance was introduced in~\cite{Wang2019dichotomy}.
 \item[(ii)] From the view point of control theory, $x\in \mbox{EI}(Q)$ means that $x$ can be stabilized robustly to
             any neighborhood of $Q$; $x\in \mbox{EIM}(Q)$ means that $x$ can be stabilized robustly to
             any neighborhood of $Q$ in the mean; $x\in \mbox{MEI}(Q)$ means that $x$ can be stabilized robustly to
             any neighborhood of $Q$ eventually in the mean.
 \item[(iii)] There are the following implication relations among these six types of equi-invariability (Figure~1).
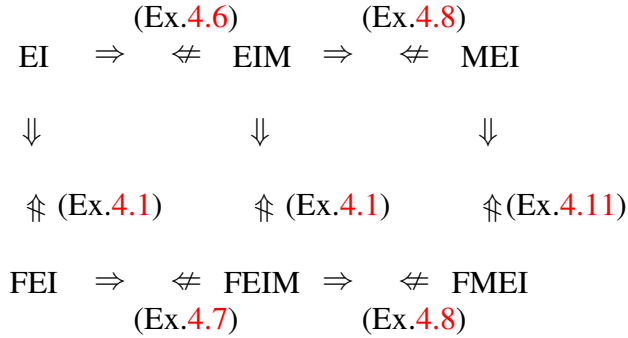
\begin{figure}[H]\label{fig:1}
  \centering
  \begin{tikzpicture}
\node at (0,0) {EI};
\node at (3,0) {EIM};
\node at (6,0) {MEI};
\node at (0,-3) {FEI};
\node at (3,-3) {FEIM};
\node at (6,-3) {FMEI};
\node at (1,0) {$\Rightarrow$};
\node at (2,0.5) {(Ex.\ref{Exa:EIM-n-EI})};
\node at (2,0) {$\nLeftarrow$};
\node at (4,0) {$\Rightarrow$};
\node at (5,0.5) {(Ex.\ref{Exa:MEI-n-FEIM})};
\node at (5,0) {$\nLeftarrow$};
\node at (1,-3) {$\Rightarrow$};
\node at (2,-3.5) {(Ex.\ref{Exa:EIM-n-FEI})};
\node at (2,-3) {$\nLeftarrow$};
\node at (4,-3) {$\Rightarrow$};
\node at (5,-3.5) {(Ex.\ref{Exa:MEI-n-FEIM})};
\node at (5,-3) {$\nLeftarrow$};
\node at (1,-2) {(Ex.\ref{Exa:FEI-n-EI})};
\node[rotate=-90] at (0,-2) {$\nLeftarrow$};
\node[rotate=90] at (0,-1) {$\Leftarrow$};
\node at (4,-2) {(Ex.\ref{Exa:FEI-n-EI})};
\node[rotate=-90] at (3,-2) {$\nLeftarrow$};
\node[rotate=90] at (3,-1) {$\Leftarrow$};
\node at (7,-2) {(Ex.\ref{Exa:FMEI-n-MEI})};
\node[rotate=-90] at (6,-2) {$\nLeftarrow$};
\node[rotate=90] at (6,-1) {$\Leftarrow$};
\end{tikzpicture}
  \caption{Six types of equi-invariability}
\end{figure}

We give examples in Appendix to show the above seven ``$\nRightarrow$" relations are possible.
We establish in the next section the conditions under which EI$\Leftrightarrow$FEI, EIM$\Leftrightarrow$FEIM and
MEI$\Leftrightarrow$FMEI, respectively. See Corollaries~\ref{cor3.3}, \ref{cor3.9} and \ref{cor3.15}.
\end{itemize}
\end{rem}

Now let us discuss the relations between equi-invariability and the control complexity. It is well known that invariance entropy introduced by Colonius and Kawan~\cite{Colonius-Kawan2009} as well as topological feedback entropy introduced by Nair et al.~\cite{Nair2004Topological} characterizes the minimal data rate for making a subset of the state space invariant. It is a very useful invariant to describe the exponential growth rate of the minimal number of different control functions sufficient for orbits to stay in a given set when starting in a subset of this set.

 For a control system $\Sigma=(\mathbb{N}_0,X,U,\mathscr{U},\phi)$. Let a subset $Q\subset X$ be controlled invariant. For $\omega\in\mathscr{U}$, $n\in\mathbb{N}$, and $\varepsilon>0$, define
\[Q_{n,\omega}^\varepsilon=\{x\in Q:\phi([0,n),x,\omega)\subset B_\varepsilon(Q)\}.\]
A subset $F\subset\mathscr{U}$ is called \emph{$(n,\varepsilon,Q)$-spanning set} if
\[Q=\cup_{\omega\in F}Q_{n,\omega}^\varepsilon.\]
Let
\[r_{inv}(n,\varepsilon,Q)=\inf\{\sharp F: F~\text{is an}~(n,\varepsilon,Q)\text{-spanning set}\},\]
where $\sharp F$ denotes the cardinality of $F$.

Recall that the \emph{outer invariance entropy} of $Q$ is defined by
\begin{equation}\label{eq:out entropy}
h_{inv,out}(Q):=\lim_{\varepsilon\to0}\limsup_{n\to\infty}\frac{\log r_{inv}(n,\varepsilon,Q)}{n}.
\end{equation}
See the monograph~\cite{Kawan2013} for more details on invariance entropy.

If a control invariant set $Q$ has positive out invariance entropy, that is, $h_{inv,out}(Q)>0$, then the numbers of controls
needed to keep $Q$ in any neighborhood of $Q$ in $[0, n]$ grow exponentially with respect to $n$. Thus it is more difficult to realize
such control task in this case. On the contrary, such control task is simple if $h_{inv,out}(Q)=0$. A particular case of
$h_{inv,out}(Q)=0$ is the following.

\begin{defn}(\cite[Definition 3.3]{Wang2019dichotomy})
We say that $Q$ has \emph{bounded invariance complexity} if for any $\varepsilon>0$, there exists $C:=C(\varepsilon)>0$ such that $r_{inv}(n,\varepsilon,Q)\leq C$ for all $n\in\mathbb{N}$.
\end{defn}

It is easy to see that if $Q$ has bounded invariance complexity then its outer invariance entropy is zero.
For some $A\subset Q$, we denote the closure of $A$ in $Q$ with respect to the subspace
topology on $Q$ by $cl_Q A$.

\begin{prop}(\cite[Proposition~3.5]{Wang2019dichotomy})
Let $\Sigma=(\mathbb{N}_0,X,U,\mathscr{U},\phi)$ be a control system and
$Q$ be a compact subset of $X$. Then
\begin{itemize}
\item[(i)] If $Q$ is equi-invariant, then $Q$ has bounded topological invariant
complexity.
\item[(ii)] Conversely, if $U$ is a compact metrizable space and $\phi: \mathbb{N}_0\times X\times\mathscr{U}\rightarrow X$
 is continuous, then the bounded topological invariant complexity of $Q$ implies that $cl_Q \mbox{EI}(Q) = Q$.
\end{itemize}
\end{prop}

\begin{thm}\label{thm:BIC-iff-FEI}
Let $\Sigma=(\mathbb{N}_0,X,U,\mathscr{U},\phi)$ be a system and $Q\subset X$ be a nonempty compact set. Then $Q$ is finitely equi-invariant if and only if $Q$ has bounded invariance complexity.
\end{thm}

\begin{proof}
($\Rightarrow$) Suppose that $Q$ is finitely equi-invariant.
Then for any $\varepsilon>0$ and $x\in Q$, there are $\delta_x>0$ and $F_x\subset\mathscr{U}$ such that for every $y\in B(x,\delta_x)\cap Q$ there exists $\omega_y\in F_x$ such that $\phi(\mathbb{N}_0,y,\omega_y)\subset B_\varepsilon(Q)$.
Since $Q$ is compact, we can find a finite open cover $\mathcal {C}:=\{B(x_i,\delta_{x_i}),i=1,\ldots,p\}$ of $Q$. Let $F=\cup_{i=1}^p\{F_{x_i}\}$. Then $F$ is finite and is an $(n,\varepsilon,Q)$-spanning set for every $n\in\mathbb{N}$. Hence $Q$ has bounded invariance complexity.


($\Leftarrow$)
Given $\varepsilon>0$, there exists $C$ such that $r_{inv}(n,\frac{\varepsilon}{3},Q)\leq C$ for all $n\in\mathbb{N}$; that is, for any $n\in\mathbb{N}$ there exists $F_n\subset\mathscr{U}$ such that $\sharp F_n\leq C$ and $Q=\cup_{\omega\in F_n}Q_{n,\omega}^{\varepsilon/3}$. By the compactness of $2^\mathscr{U}$(the hyperspace of $\mathscr{U}$~\cite{Nadler1978Hyperspaces}), we can pick a convergent subsequence $\{F_{n_i}\}$ in $\{F_n\}$. We denote its limit as $F$; that is $\lim\limits_{i\to\infty} F_{n_i}=F$. Therefore, we have $\sharp F\leq C$ by the fact that $\{A\in2^\mathscr{U}:\sharp A\leq C\}$ is closed. For every $i\in\mathbb{N}$ and any $x\in Q$ there exists $\omega_{n_i}\in F_{n_i}$ such that
\[\phi([0,n_i),x,\omega_{n_i})\subset B_{\varepsilon/3}(Q).\]
Thus we get
\[\phi([0,n_i),x,\omega_{n_j})\subset B_{\varepsilon/3}(Q)\]
for any $j\geq i$. Suppose that $\omega_{n_i}\to\omega$. Then $\omega\in F$. Letting $j\to\infty$, we have, by the continuity of $\phi$,
\[\phi([0,n_i),x,\omega)\subset B_{\varepsilon/2}(Q).\]
Since $n_i\to\infty$ as $i\to\infty$, we obtain $\phi(\mathbb{N}_0,x,\omega)\subset B_{\varepsilon/2}(Q)$. This implies that
\[Q\subset\bigcup_{\omega\in F}\bigcap_{n=1}^\infty Q_{n,\omega}^{\varepsilon/2}.\]
It follows that $Q$ is finitely equi-invariant.
\end{proof}

Next, we discuss the relations between finite equi-invariance in the mean and bounded invariance complexity in the mean.

Given $\omega\in\mathscr{U}$, $n\in\mathbb{N}$, and $\varepsilon>0$, let
\[\hat{Q}_{n,\omega}^\varepsilon=\{x\in Q:\max_{1\leq k\leq n}\{\frac1{k}\sum_{i=0}^{k-1}d[\phi(i,x,\omega),Q]\}<\varepsilon\}.\]
A subset $F\subset\mathscr{U}$ is called \emph{$(n,\varepsilon,Q)$-spanning set in the mean} if
\[Q=\cup_{\omega\in F}\hat{Q}_{n,\omega}^\varepsilon.\]
Let
\[\hat{r}_{inv}(n,\varepsilon,Q)=\inf\{\sharp F: F~\text{is an}~(n,\varepsilon,Q)\text{-spanning set in the mean}\}.\]

\begin{defn}
We say that $Q$ has \emph{bounded invariance complexity in the mean} if for any $\varepsilon>0$, there exists $C:=C(\varepsilon)>0$ such that $\hat{r}_{inv}(n,\varepsilon,Q)\leq C$ for all $n\in\mathbb{N}$.
\end{defn}

\begin{thm}\label{thm:BICIM-iff-FEIIM}
Let $\Sigma=(\mathbb{N}_0,X,U,\mathscr{U},\phi)$ be a system and $Q\subset X$ be a nonempty compact set. Then $Q$ is finitely equi-invariant in the mean if and only if $Q$ has bounded invariance complexity in the mean.
\end{thm}

\begin{proof}
($\Rightarrow$) Suppose that $Q$ is finitely equi-invariant. Then for any $\varepsilon>0$ and $x\in Q$,
there are $\delta_x>0$ and $F_x\subset\mathscr{U}$ such that for every $y\in B(x,\delta_x)\cap Q$ there exists $\omega_y\in F_x$ such that
\[\frac1{n}\sum_{i=0}^{n-1}d[\phi(i,y,\omega_y),Q]<\varepsilon,~\forall~n\in\mathbb{N}.\]
Since $Q$ is compact, we can find a finite open cover $\mathcal{C}:=\{B(x_i,\delta_{x_i}),i=1,\ldots,p\}$ of $Q$. Let $F=\cup_{i=1}^pF_{x_i}$.
Then $F$ is finite and is an $(n,\varepsilon,Q)$-spanning set in the mean for every $n\in\mathbb{N}$. Hence $Q$ has bounded invariance complexity in the mean.


($\Leftarrow$)
Given $\varepsilon>0$, there exists $C$ such that $\hat{r}_{inv}(n,\frac{\varepsilon}{3},Q)\leq C$ for all $n\in\mathbb{N}$; that is, for any $n\in\mathbb{N}$ there exists $F_n\subset\mathscr{U}$ such that $\sharp F_n\leq C$ and $Q=\cup_{\omega\in F_n}\hat{Q}_{n,\omega}^{\varepsilon/3}$. By the compactness of $2^\mathscr{U}$, we can pick a convergent subsequence $\{F_{n_i}\}$ in $\{F_n\}$. We denote its limit as $F$. Therefore, we have $\sharp F\leq C$ by the fact that $\{A\in 2^\mathscr{U}:\sharp A\leq C\}$ is closed. For every $i\in\mathbb{N}$ and any $x\in Q$ there exists $\omega_{n_i}\in F_{n_i}$ such that
\[\max_{1\leq k\leq n_i}\{\frac1{k}\sum_{i=0}^{k-1}d[\phi(i,x,\omega_{n_i}),Q]\}<\frac{\varepsilon}{3}.\]
Thus we get
\[\max_{1\leq k\leq n_i}\{\frac1{k}\sum_{i=0}^{k-1}d[\phi(i,x,\omega_{n_j}),Q]\}<\frac{\varepsilon}{3}.\]
for any $j\geq i$. Suppose that $\omega_{n_i}\to\omega$. Then $\omega\in F$. Letting $j\to\infty$, we have, by the continuity of $\phi$,
\[\max_{1\leq k\leq n_i}\{\frac1{k}\sum_{i=0}^{k-1}d[\phi(i,x,\omega),Q]\}<\frac{\varepsilon}{2}.\]
Since $n_i\to\infty$ as $i\to\infty$, we obtain
\[\max_{1\leq k\leq n}\{\frac1{k}\sum_{i=0}^{k-1}d[\phi(i,x,\omega),Q]\}<\frac{\varepsilon}{2},~\forall~n\in\mathbb{N}.\]
This implies that
\[Q\subset\bigcup_{\omega\in F}\bigcap_{n=1}^\infty \hat{Q}_{n,\omega}^{\frac{\varepsilon}{2}}.\]
It follows that $Q$ is finitely equi-invariant in the mean.
\end{proof}

Finally in this section, we characterize the concept of finitely mean equi-invariance by finitely mean stability of $Q$ in the sense of Lyapunov.

Let $E\subset\N_0$. We define the \emph{upper density} $\overline{D}(E)$ of $E$ by
\[\overline{D}(E)=\limsup_{n\to\infty}\frac{\sharp (E\cap[0,n-1])}{n}.\]

\begin{defn}
Let $\Sigma=(\mathbb{N}_0,X,U,\mathscr{U},\phi)$ be a system and $Q\subset X$ be a nonempty set. A point $x\in Q$ is said to be
\emph{finitely mean stable point of $Q$ in the sense of Lyapunov} (abbreviated as \emph{finitely mean-L-stable point of $Q$})
if for every $\varepsilon>0$ there exist $\delta>0$ and a finite subset $F\subset\mathscr{U}$ such that $y\in B(x,\delta)\cap Q$ implies $d(\phi(n,y,\omega),Q)<\varepsilon$ for some $\omega\in F$ and all $n\in\mathbb{N}_0$ except a set of upper density less than $\varepsilon$. We call $Q$ \emph{finitely mean-L-stable} if every $x\in Q$ is a finitely mean-L-stable point of $Q$.
\end{defn}

\begin{thm}\label{thm:L-stable-equ-inv}
Let $\Sigma=(\mathbb{N}_0,X,U,\mathscr{U},\phi)$ be a system with $\diam(X)<\infty$ and $Q\subset X$ be a nonempty set.
Then $Q$ is finitely mean-L-stable if and only if it is finitely mean equi-invariant.
\end{thm}

\begin{proof}
($\Leftarrow$) Suppose that $Q$ is finitely mean equi-invariant.
Then for any $x\in Q$ and $\varepsilon>0$, there exist $\delta>0$ and a finite subset $F\subset\mathscr{U}$
 such that for every $y\in Q$ with $d(x,y)<\delta$, we have
\[\limsup_{n\to\infty}\frac1{n}\sum_{i=0}^{n-1}d(\phi(i,y,\omega),Q)<\varepsilon^2,\]
for some $\omega\in F$.
Let $E=\{k\in\mathbb{N}_0: d(\phi(k,y,\omega),Q)\geq\varepsilon\}$. Thus
\[\varepsilon^2>\limsup_{n\to\infty}\frac1{n}\sum_{i=0}^{n-1}d(\phi(i,y,\omega),Q)\geq\limsup_{n\to\infty}\frac1{n}(\varepsilon\cdot\sharp([0,n-1]\cap E))=\varepsilon\cdot\overline{D}(E).\]
It follows that $\overline{D}(E)<\varepsilon$. Therefore, $Q$ is finitely mean-L-stable.

($\Rightarrow$) Assume that $Q$ is finitely mean-L-stable. For any $x\in Q$ and $\varepsilon>0$, let
\[\eta=\frac{\varepsilon}{2(\diam X+1)}.\]
Then there exist $\delta>0$ and  a finite subset $F\subset\mathscr{U}$ such that for any $y\in B(x,\delta)\cap Q$, $d(\phi(n,y,\omega),Q)<\eta$
for some $\omega\in F$ and  all $n\in\mathbb{N}_0$ except a set of upper density less than $\eta$. Let
\[E=\{k\in\mathbb{N}_0:d(\phi(n,y,\omega),Q)\geq\eta\}.\]
Thus $\overline{D}(E)<\eta$ and
\begin{align*}
\limsup_{n\to\infty}\frac1{n}\sum_{i=0}^{n-1}d(\phi(i,y,\omega),Q)&\leq\limsup_{n\to\infty}\frac1{n}(\diam(X)\cdot\sharp([0,n-1]\cap E)+n\eta)\\
&\leq\diam(X)\overline{D}(E)+\eta\leq\diam(X)\eta+\eta\leq\frac{\varepsilon}{2}.
\end{align*}
This implies that $Q$ is finitely mean equi-invariant.
\end{proof}

\section{Dichotomy theorems  for control sets}\label{sec:Dichotomy-theorem}
In this section, we will discuss three types of dichotomy theorems for control sets.

First, let us recall some basic notions.
Let $\Sigma=(\mathbb{N}_0,X,U,\mathscr{U},\phi)$ be a control system. For $x\in X$ and $n\in\mathbb{N}$, the \emph{set of points reachable from $x$ up to time $n$} is defined by
\[\mathcal{O}_{\leq n}^{+}(x):=\{y\in X : \exists  m\in[0,n], \omega\in\mathscr{U}~\text{with} ~y=\phi(m,x,\omega)\}.\]
The \emph{positive orbit} of $x$ is defined by
\[\mathcal{O}^+(x)=\bigcup_{n\in\mathbb{N}}\mathcal{O}_{\leq n}^{+}(x).\]

\begin{defn}(see~\cite[Definition 3.1]{Colonius1993Some} or~\cite[Definition 1.12]{Kawan2013})
Let $\Sigma=(\mathbb{N}_0,X,U,\mathscr{U},\phi)$ be a system. A set $D\subset X$ is called a \emph{control set} of system $\Sigma$ if the following conditions hold.
\begin{enumerate}
\item $D$ is controlled invariant, that is, for every $x\in D$ there exists $\omega\in \mathscr{U}$ such that $\phi(\N_0, x,\omega)\subset D$.

\item For all $x\in D$ one has $D\subset cl~\mathcal{O}^{+}(x)$ , where $cl~\mathcal{O}^{+}(x)$ denotes the closure of $\mathcal{O}^{+}(x)$.

\item $D$ is maximal with these properties, that is, if $D'\supset D$ satisfies conditions (1) and (2), then $D'=D$.
\end{enumerate}
\end{defn}

\subsection{The first type of dichotomy theorem}
By Theorem~\ref{thm:BIC-iff-FEI} and Corollary 3.7 in~\cite{Wang2019dichotomy}, we have
\begin{cor}\label{cor3.3}
Let $\Sigma=(\mathbb{N}_0,X,U,\mathscr{U},\phi)$ and $Q\subset X$ be a compact control set with nonempty interior. Assume that $\mathscr{U}$ is a compact metrizable space and $\phi: \mathbb{N}_0\times X\times\mathscr{U}\to X$ is continuous. Then the following conditions are equivalent.
\begin{enumerate}
  \item $Q$ is equi-invariant;
  \item $Q$ is finitely equi-invariant;
  \item $Q$ has bounded invariance complexity.
\end{enumerate}
\end{cor}

\begin{defn}
Let $\Sigma=(\mathbb{N}_0,X,U,\mathscr{U},\phi)$ be a system and $Q\subset X$ be a nonempty set. We say that $Q$ is an \emph{unstable set} if there exists $\varepsilon>0$ such that for any $x\in Q$, $\delta>0$ and $\omega\in\mathscr{U}$, we have
\[d[\phi(m,y,\omega),Q]\geq\varepsilon\]
for some $y\in B(x,\delta)\cap Q$ and $m\in\mathbb{N}_0$.
\end{defn}

In~\cite{Wang2019dichotomy}, Wang et al. showed the following dichotomy theorem for control sets.

\begin{thm}(\cite[Theorem 3.13]{Wang2019dichotomy})
Let $\Sigma=(\mathbb{N}_0,X,U,\mathscr{U},\phi)$ be a system and $Q\subset X$ be a control set with $cl \text{Int}(Q) =cl Q$. Then $Q$ is either equi-invariant or unstable.
\end{thm}

\subsection{The second type of dichotomy theorem}
Let
\begin{align*}
EIM_k(Q)=\{x\in Q:&\exists~\delta>0~\text{and}~\omega\in\mathscr{U}~\text{s.t.}\\
&\frac1{n}\sum_{i=0}^{n-1}d[\phi(i,y,\omega),Q]<\frac1{k},~\forall~n\in\mathbb{N}~\text{and}~y\in B(x,\delta)\cap Q\}.
\end{align*}
Then $EIM_k(Q)$ is an open subset of $Q$ and $EIM(Q)=\cap_{k=1}^\infty EIM_k(Q)$.

\begin{lem}\label{lem:no-return}(\cite[Corollary 1.1]{Kawan2013})
A control set $D$  with nonempty interior has the no-return property, that is, if $x\in D, n\in \mathbb{N}_0$ and $\omega\in \mathscr{U}$ with $\phi(n,x,\omega)\in D$ implies $\phi([0,n], x,\omega)\subset D$.
\end{lem}

\begin{lem}\label{lem:EIMk(Q)-equials-Q}
Let $\Sigma=(\mathbb{N}_0,X,U,\mathscr{U},\phi)$ be a system and $Q\subset X$ be a control set.
If $EIM_k(Q)\cap\text{Int}(Q) \neq\emptyset$ for some $k\in\mathbb{N}$, then $EIM_k(Q)=Q$.
\end{lem}

\begin{proof}
Pick $x\in EIM_k(Q)\cap \text{Int}(Q) $. Then there exist $\delta>0$ and $\omega\in\mathscr{U}$ such that for every $y\in B(x,\delta)\subset\text{Int}(Q) $, we have
\[\frac1{n}\sum_{i=0}^{n-1}d[\phi(i,y,\omega),Q]<\frac1{k},~\forall~n\in\mathbb{N}.\]
For any $x'\in Q$, there exist $m\in\mathbb{N}_0$, $\omega'\in\mathscr{U}$, and $\delta'>0$ such that
\[\phi(m,B(x',\delta'),\omega')\subset B(x,\delta)\subset\text{Int}(Q) .\]
By the no-return property (see Lemma~\ref{lem:no-return}),
\[\phi([0,m],B(x',\delta'),\omega')\subset B(x,\delta)\subset\text{Int}(Q) .\]
Let $\hat{\omega}=\omega'\omega^m$. Then for any $y\in B(x',\delta')$,
\[\frac1{n}\sum_{i=0}^{n-1}d[\phi(i,y,\hat{\omega}),Q]=\left\{
    \begin{array}{ll}
      =0, & 0\leq n\leq m, \\
      <\frac1{k}, & n>m.
    \end{array}
  \right.
\]
So $x'\in EIM_k(Q)$ and $EIM_k(Q)=Q$.
\end{proof}

By Theorem~\ref{thm:BICIM-iff-FEIIM}, we have
\begin{cor}\label{cor3.9}
Let $\Sigma=(\mathbb{N}_0,X,U,\mathscr{U},\phi)$ and $Q\subset X$ a compact control set with nonempty interior. Assume that $\mathscr{U}$ is a compact metrizable space and $\phi: \mathbb{N}_0\times X\times\mathscr{U}\to X$ is continuous. Then the following conditions are equivalent.
\begin{enumerate}
  \item $Q$ is equi-invariant in the mean;
  \item $Q$ is finitely equi-invariant in the mean;
  \item $Q$ has bounded invariance complexity in the mean.
\end{enumerate}
\end{cor}

\begin{proof}
We have shown that $(2)$ and $(3)$ are equivalent.
It is clear that $(1)$ implies $(2)$, we only need to prove $(3)$ implies $(1)$.
For every $k\in\mathbb{N}$, it follows from the proof of~Theorem~\ref{thm:BICIM-iff-FEIIM} that
\[Q\subset\bigcup_{\omega\in F}\bigcap_{n=1}^\infty \hat{Q}_{n,\omega}^{1/2k},\]
where $F$ is a finite set of $\mathscr{U}$ and
\[\hat{Q}_{n,\omega}^{1/2k}=\{x\in Q:\max_{1\leq j\leq n}\{\frac1{j}\sum_{i=0}^{j-1}d[\phi(i,x,\omega),Q]\}<\frac{1}{2k}\}.\]
Let $F=\{\omega_i:1\leq i\leq\sharp F\}$ and
\[Q_i=\cap_{n=1}^\infty\{x\in Q:\max_{1\leq j\leq n}\{\frac1{j}\sum_{r=0}^{j-1}d[\phi(r,x,\omega_i),Q]\leq\frac{1}{2k}\},i=1,\ldots,\sharp F.\]
Then $Q_i$ is closed in $Q$ for $i=1,\ldots,\sharp F$ and $Q=\cup_{i=1}^{\sharp F}Q_i$.
Let $Q_1':=Q_1$ and $Q_i':=cl_Q(Q_i\setminus\cup_{i=1}^{i-1}Q_j)$ for $2\leq i\leq\sharp F$.
Using Lemma 3.4 in~\cite{Wang2019dichotomy}, we have
\[\bigcup_{i=1}^{\sharp F}Q_i'=Q,~\bigcup_{i=1}^{\sharp F}cl_Q(Q_i'\setminus\cup_{j\neq i}Q_j')=Q.\]
For any $i\in\{1,\ldots,\sharp F\}$ and $x\in Q_i'\setminus\cup_{j\neq i}Q_j'$, there exists $\delta>0$ such that
\[B(x,\delta)\cap Q=B(x,\delta)\cap Q_i'.\]
Therefore, we get
\begin{align*}
B(x,\delta)\cap Q\subset Q_i'\subset Q_i,
\end{align*}
which implies that
\begin{align*}
\frac1{n}\sum_{r=0}^{n-1}d[\phi(r,y,\omega_i),Q]<\frac1{k},~\forall~n\in\mathbb{N}~\text{and}~y\in B(x,\delta)\cap Q.
\end{align*}
Thus
$Q_i'\setminus\cup_{j\neq i}Q_j'\subset EIM_k(Q)$.
It follows that $cl_Q EIM_k(Q)=Q$. According to the Baire category theorem,
we see that $EIM(Q)=\cap_{k=1}^\infty EIM_k(Q)$ is a dense $G_\delta$ subset of $Q$. Hence $EIM(Q)\cap \text{Int}(Q) \neq\emptyset$.
Pick $x\in EIM(Q)\cap \text{Int}(Q) $. Using Lemma~\ref{lem:EIMk(Q)-equials-Q}, we obtain
$EIM(Q)=Q$.
\end{proof}

\begin{defn}
Let $\Sigma=(\mathbb{N}_0,X,U,\mathscr{U},\phi)$ be a system and $Q\subset X$ be a nonempty set. We say that $Q$ is an \emph{unstable set in the mean} if there exists $\varepsilon>0$ such that for any $x\in Q$, $\delta>0$ and $\omega\in\mathscr{U}$ there exist $y\in B(x,\delta)\cap Q$ and $m\in\mathbb{N}_0$,
\[\frac1{m}\sum_{i=0}^{m-1}d[\phi(i,y,\omega),Q]\geq\varepsilon.\]
\end{defn}
By a direct observation, we get
\begin{lem}\label{lem:unstable-mean-iff-EIMK}
Let $\Sigma=(\mathbb{N}_0,X,U,\mathscr{U},\phi)$ be a system. Then $Q\subset X$ is an unstable set in the mean if and only if there exists $k\in\mathbb{N}$ such that $EIM_k(Q)=\emptyset$.
\end{lem}

\begin{thm}
Let $\Sigma=(\mathbb{N}_0,X,U,\mathscr{U},\phi)$ be a system and $Q\subset X$ be a control set with $cl \text{Int}(Q) =cl Q$. Then $Q$ is either equi-invariant in the mean or unstable in the mean.
\end{thm}

\begin{proof}
If $Q=EIM(Q)$ then $Q$ is equi-invariant in the mean. If $Q\neq EIM(Q)$, then there exists $k\in\mathbb{N}$ such that $EIM_k(Q)\cap\text{Int}(Q) =\emptyset$ by Lemma~\ref{lem:EIMk(Q)-equials-Q}. To obtain a contradiction, we suppose that $Q$ is not unstable in the mean. By Lemma~\ref{lem:unstable-mean-iff-EIMK}, we have $EIM_k(Q)\neq\emptyset$ for all $k\in\mathbb{N}$. Fix any $k\in\mathbb{N}$ and pick $x\in EIM_k(Q)$. Then there exists $\delta>0$ and $\omega\in\mathscr{U}$ such that
\[\frac1{n}\sum_{i=0}^{n-1}d[\phi(i,y,\omega),Q]<\frac1{k},~\forall~n\in\mathbb{N}~\text{and}~y\in B(x,\delta)\cap Q.\]
Noting that $cl \text{Int}(Q) =cl Q$, we have $B(x,\delta)\cap\text{Int}(Q) \neq\emptyset$. Hence there exist $y\in B(x,\delta)\cap\text{Int}(Q) $ and $\delta'>0$ such that $B(y,\delta')\subset B(x,\delta)\cap\text{Int}(Q) $. So
\[\frac1{n}\sum_{i=0}^{n-1}d[\phi(i,y,\omega),Q]<\frac1{k},~\forall~n\in\mathbb{N}~\text{and}~y\in B(y,\delta')\cap Q  \subset B(x,\delta)\cap Q.\]
This implies that $y\in EIM_k(Q)\cap\text{Int}(Q) $ for all $k\in\mathbb{N}$, which is a contradiction.
\end{proof}

\subsection{The third type of dichotomy theorem}
Finally, we discuss the dichotomy theorem based on mean equi-invariability.

Let
\begin{align*}
\mbox{MEI}_k(Q)=\{x\in Q:&\exists~\delta>0~\text{and}~\omega\in\mathscr{U}~\text{s.t.}\\
&\limsup_{n\rightarrow\infty}\frac1{n}\sum_{i=0}^{n-1}d[\phi(i,y,\omega),Q]<\frac1{k},~\forall~y\in B(x,\delta)\cap Q\}.
\end{align*}
Then $\mbox{MEI}_k(Q)$ is an open subset of $Q$ and $\mbox{MEI}(Q)=\cap_{k=1}^\infty \mbox{MEI}_k(Q)$.

\begin{lem}\label{lem:3-13}
Let $\Sigma=(\mathbb{N}_0,X,U,\mathscr{U},\phi)$ be a system and $Q\subset X$ be a control set.
If $\mbox{MEI}_k(Q)\cap\text{Int}(Q) \neq\emptyset$ for some $k\in\mathbb{N}$, then $\mbox{MEI}_k(Q)=Q$.
\end{lem}

\begin{proof}
Pick $x\in \mbox{MEI}_k(Q)\cap \text{Int}(Q) $. Then there exist $\delta>0$ and $\omega\in\mathscr{U}$ such that for every $y\in B(x,\delta)\subset\text{Int}(Q) $, we have
\[\limsup_{n\rightarrow\infty}\frac1{n}\sum_{i=0}^{n-1}d[\phi(i,y,\omega),Q]<\frac1{k}.\]
For any $x'\in Q$, there exist $m\in\mathbb{N}_0$, $\omega'\in\mathscr{U}$, and $\delta'>0$ such that
\[\phi(m,B(x',\delta'),\omega')\subset B(x,\delta)\subset\text{Int}(Q) .\]
Applying the no-return property again, we have
\[\phi([0,m],B(x',\delta'),\omega')\subset B(x,\delta)\subset\text{Int}(Q).\]
Let $\hat{\omega}=\omega'\omega^m$. Then for any $y\in B(x',\delta')$,
\begin{align*}
\limsup_{n\rightarrow\infty,n>m}\frac1{n}\sum_{i=0}^{n-1}d[\phi(i,y,\hat{\omega}),Q]
=&\limsup_{n\rightarrow\infty,n>m}\frac1{n}\LARGE(\sum_{i=0}^{m}d[\phi(i,y,\hat{\omega}),Q]+\sum_{i=m+1}^{n-1}d[\phi(i,y,\hat{\omega}),Q]\LARGE)\\
=&\limsup_{n\rightarrow\infty,n>m}\frac1{n}\sum_{i=m+1}^{n-1}d[\phi(i,\phi(m,y,\hat{\omega}),\omega),Q])\\
<&\frac{1}{k}.
\end{align*}
So $x'\in \mbox{MEI}_k(Q)$ and $\mbox{MEI}_k(Q)=Q$.
\end{proof}

\begin{cor}\label{cor3.15}
Let $\Sigma=(\mathbb{N}_0,X,U,\mathscr{U},\phi)$ and $Q\subset X$ a compact control set with nonempty interior. Assume that $\mathscr{U}$ is a compact metrizable space and $\phi: \mathbb{N}_0\times X\times\mathscr{U}\to X$ is continuous. Then the following conditions are equivalent.
\begin{enumerate}
  \item $Q$ is mean equi-invariant;
  \item $Q$ is finitely  mean equi-invariant;
  \item $Q$ is finitely mean-L-stable.
\end{enumerate}
\end{cor}

\begin{proof}
We have show that $(2)$ and $(3)$ are equivalent in Theorem~\ref{thm:L-stable-equ-inv}.
It is clear that $(1)$ implies $(2)$, we only need to prove $(2)$ implies $(1)$.

For every $\epsilon>0$ and $\omega\in\mathscr{U}$, denote
\[\widetilde{Q}_{\omega}^{\varepsilon}:=\{x\in Q:\limsup_{n\rightarrow\infty}\frac1{n}\sum_{i=0}^{n-1}d[\phi(i,x,\omega),Q]\leq\frac{\varepsilon}{2}\}.\]
Let $\epsilon>0$. For every $x\in Q$, by finite  mean equi-invariance of $Q$,
there exist $\delta_x>0$ and $F_x$ such that for any $y\in B(x,\delta_x)\cap Q$, there holds
\[\limsup_{n\rightarrow\infty}\frac1{n}\sum_{i=0}^{n-1}d[\phi(i,x,\omega),Q]<\frac{\epsilon}{2}\]
for some $\omega\in F_x$.
Take $\mathcal {C}:=\{B(x,\delta_x)\}_{x\in Q}$.
Then $\mathcal {C}$ is an open cover of $Q$.
By compactness of $Q$, there exist finite open balls
$$B(x_1,\delta_{x_1}), B(x_2,\delta_{x_2}),\ldots, B(x_r,\delta_{x_r})$$
such that $Q\subset \cup_{i=1}^rB(x_i,\delta_{x_i})$.
Put $F=\cup_{i=1}^rF_{x_i}:=\{\omega^{(i)}:1\leq i\leq\sharp F\}$.
Then $\widetilde{Q}_{\omega^{(i)}}^{\varepsilon}$ is closed in $Q$ for $i=1,\ldots,\sharp F$
 and $Q=\cup_{i=1}^{\sharp F}\widetilde{Q}_{\omega^{(i)}}^{\varepsilon}$.
Let $Q_1':=\widetilde{Q}_{\omega^{(1)}}^{\varepsilon}$ and $Q_i':=cl_Q(\widetilde{Q}_{\omega^{(i)}}^{\varepsilon}\setminus\cup_{j=1}^{i-1}\widetilde{Q}_{\omega^{(j)}}^{\varepsilon})$ for $2\leq i\leq\sharp F$.
Using Lemma 3.4 in~\cite{Wang2019dichotomy} again, we have
\[\bigcup_{i=1}^{\sharp F}Q_i'=Q,~\bigcup_{i=1}^{\sharp F}cl_Q(Q_i'\setminus\cup_{j\neq i}Q_j')=Q.\]
For each $i\in\{1,\ldots,\sharp F\}$ and $x\in Q_i'\setminus\cup_{j\neq i}Q_j'$, there exists $\delta>0$ such that
\[B(x,\delta)\cap Q=B(x,\delta)\cap Q_i'.\]
Therefore, we get
\begin{align*}
B(x,\delta)\cap Q\subset Q_i'\subset \widetilde{Q}_{\omega^{(i)}}^{\varepsilon},
\end{align*}
which implies that
\begin{align*}
\limsup_{n\rightarrow\infty}\frac1{n}\sum_{r=0}^{n-1}d[\phi(r,y,\omega^{(i)}),Q]<\epsilon,~\text{for all}~y\in B(x,\delta)\cap Q.
\end{align*}
Thus
$Q_i'\setminus\cup_{j\neq i}Q_j'\subset \mbox{MEI}_k(Q)$.
It follows that $cl_Q \mbox{MEI}_k(Q)=Q$. According to the Baire category theorem,
we see that $\mbox{MEI}(Q)=\cap_{k=1}^\infty \mbox{MEI}_k(Q)$ is a dense $G_\delta$ subset of $Q$.
Hence $\mbox{MEI}(Q)\cap \text{Int}(Q) \neq\emptyset$.
Pick $x\in \mbox{MEI}(Q)\cap \text{Int}(Q) $.
Using Lemma~\ref{lem:3-13}, we obtain
$\mbox{MEI}(Q)=Q$.
\end{proof}

\begin{defn}
Let $\Sigma=(\mathbb{N}_0,X,U,\mathscr{U},\phi)$ be a system and $Q\subset X$ be a nonempty set. We say that $Q$ is a \emph{mean unstable set} if there exists $\varepsilon>0$ such that for any $x\in Q$, $\delta>0$ and $\omega\in\mathscr{U}$ there exists $y\in B(x,\delta)\cap Q$,
\[\limsup_{n\rightarrow\infty}\frac1{n}\sum_{i=0}^{n-1}d[\phi(i,y,\omega),Q]\geq\varepsilon.\]
\end{defn}
The following Lemma is obvious.
\begin{lem}\label{lem:mean-unstable-iff-EIMK}
Let $\Sigma=(\mathbb{N}_0,X,U,\mathscr{U},\phi)$ be a system.
Then $Q\subset X$ is a mean unstable set  if and only if there exists $k\in\mathbb{N}$ such that $\mbox{MEI}_k(Q)=\emptyset$.
\end{lem}

\begin{thm}
Let $\Sigma=(\mathbb{N}_0,X,U,\mathscr{U},\phi)$ be a system and $Q\subset X$ be a control set with $cl \text{Int}(Q) =cl Q$.
Then $Q$ is either mean equi-invariant or mean unstable.
\end{thm}

\begin{proof}
If $Q=\mbox{MEI}(Q)$ then $Q$ is mean equi-invariant.
If $Q\neq \mbox{MEI}(Q)$,
then there exists $k\in\mathbb{N}$ such that $\mbox{MEI}_k(Q)\cap\text{Int}(Q) =\emptyset$ by Lemma~\ref{lem:3-13}.
Suppose in contrast that $Q$ is not mean unstable. By Lemma~\ref{lem:mean-unstable-iff-EIMK},
we have $\mbox{MEI}_k(Q)\neq\emptyset$ for all $k\in\mathbb{N}$.
Fix any $k\in\mathbb{N}$ and pick $x\in \mbox{MEI}_k(Q)$.
Then there exist $\delta>0$ and $\omega\in\mathscr{U}$ such that
\[\limsup_{n\rightarrow\infty}\frac1{n}\sum_{i=0}^{n-1}d[\phi(i,y,\omega),Q]<\frac1{k},~\forall~y\in B(x,\delta)\cap Q.\]
By $cl \text{Int}(Q) =cl Q$, it follows that $B(x,\delta)\cap\text{Int}(Q) \neq\emptyset$.
Consequently, there exist $y\in B(x,\delta)\cap\text{Int}(Q) $ and $\delta'>0$ such that $B(y,\delta')\subset B(x,\delta)\cap\text{Int}(Q) $. So
\[\limsup_{n\rightarrow\infty}\frac1{n}\sum_{i=0}^{n-1}d[\phi(i,y,\omega),Q]<\frac1{k},~\forall~y\in B(y,\delta')\cap Q  \subset B(x,\delta)\cap Q,\]
which implies that $y\in \mbox{MEI}_k(Q)\cap\text{Int}(Q) $ for all $k\in\mathbb{N}$.
This is a contradiction.
\end{proof}

\section{Appendix}

\begin{exam}[FEI but not EI; FEIM but not EIM]\label{Exa:FEI-n-EI}
Consider a control system of form ~(\ref{eq:control-system-modle}), where
\begin{enumerate}
\item  $X=[0,1]=\{x\in\mathbb{R}:0\leq x \leq 1\}$;
\item  $U=\{0,1\}$;
\item  $F_0, F_1: X\rightarrow X$ are defined by
\begin{equation*}
F_0(x)=
\begin{cases}
x, & \text {if $0\leq x<\frac{3}{8},$}\\
5(x-\frac{1}{2})+1, & \text {if $\frac{3}{8}\leq x<\frac{1}{2},$}\\
1, & \text {if $\frac{1}{2}\leq x\leq 1.$}
\end{cases}
\end{equation*}
and
\begin{equation*}
F_1(x)=
\begin{cases}
1, & \text {if $0\leq x<\frac{1}{4},$}\\
-4(x-\frac{1}{4})+1, & \text {if $\frac{1}{4}\leq x<\frac{3}{8},$}\\
\frac{1}{2}, & \text {if $\frac{3}{8}\leq x<\frac{1}{2},$}\\
4(x-\frac{5}{8})+1, & \text {if $\frac{1}{2}\leq x<\frac{5}{8},$}\\
1, & \text {if $\frac{5}{8}\leq x\leq 1.$}
\end{cases}
\end{equation*}
\end{enumerate}
Let $Q=[\frac{1}{4},\frac{1}{2}]$.
Then $Q$ is finitely equi-invariant but not equi-invariant.
\end{exam}

\begin{figure}[H]
  \centering
\begin{tikzpicture}[line width=1pt,line cap=round,line join=round,x=1.0cm,y=1.0cm,scale=0.618]
\clip(-1.5,-1.5) rectangle (9,9);
\draw (0,0)node[below]{$0$}--(2,0)node[below]{$\frac14$}--(3,0)node[below]{$\frac38$}
--(4,0)node[below]{$\frac12$}--(5,0)node[below]{$\frac58$}--(8,0)node[below]{$1$}--(8,8)--(0,8)--cycle;
\draw (0,0)--(3,3)--(4,8)--(8,8);
\draw[dashed,color=red] (0,8)--(2,8)--(3,4)--(4,4)--(5,8)--(8,8);
\draw[line width=2pt,color=blue] (2,0)--(4,0);
\end{tikzpicture}
  \caption{Finitely equi-invariant.}\label{F-E-I}
\end{figure}
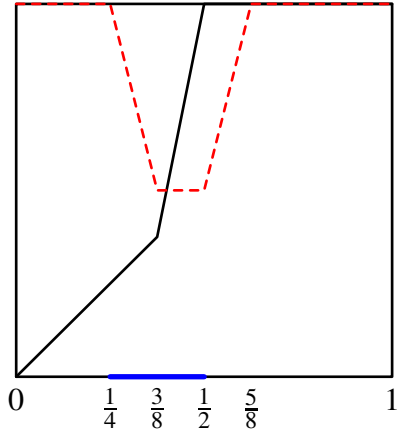

\begin{proof}
We divide our proof into three claims.

\begin{clm}\label{claim:2-1}
$Q$ is finitely equi-invariant.
\end{clm}
\begin{proof}[Proof of Claim~\ref{claim:2-1}]
Fix any $x\in Q$.
For any $\epsilon>0$, choose $\delta=\epsilon$
and $\mathcal{F}=\{\omega:=0^\infty,\widehat{\omega}:=1^\infty\}$,
then by the definitions of $F_0$ and $F_1$, for any $y\in B(x,\delta)\cap Q$ there holds:
$\phi(\mathbb{N},x,\omega)=x\in Q$ if $\frac{1}{4}\leq y\leq \frac{3}{8}$;
and $\phi(\mathbb{N},x,\widehat{\omega})=\frac{1}{2}\in Q$ if $\frac{3}{8}< y\leq \frac{1}{2}$.
\end{proof}

\begin{clm}\label{claim:2-2}
For any $\frac{1}{4}\leq x< \frac{3}{8}$ and any control sequence $\widehat{\omega}\in \{0^n1^\infty: n\geq 1\}\cup\{\omega: \omega_i\omega_{i+1}=10, \text{~for some~} i\geq 0\}\cup\{1^\infty \}$,
there holds $\lim_{n\rightarrow\infty}\phi(n,x,\widehat{\omega})=1$.
\end{clm}
\begin{proof}[Proof of Claim~\ref{claim:2-2}]
Let $y\in [\frac{1}{4}, \frac{3}{8})$ and a control sequence $\widehat{\omega}\in \{\omega: \omega_i\omega_{i+1}=10, \text{~for some~} i\geq 0\}\cup\{1^\infty \}\cup \{0^n1^\infty: n\geq 1\}$. Then we have the following cases.

Case 1. $\widehat{\omega}=1^\infty$. Then $F_{\omega_0}(y)>\frac{1}{2}$.
So by the definition of $F_1$, we have $\lim_{n\rightarrow\infty}\phi(n,x,\widehat{\omega})=1$.

Case 2. $\widehat{\omega}\in \{\omega: \omega_i\omega_{i+1}=10, \text{~for some~} i\geq 0\}$.
Note that $F_0\circ F_1(z)>\frac{1}{2}$ for all $z\in [0,1]$.
This implies that $\lim_{n\rightarrow\infty}\phi(n,x,\widehat{\omega})=1$.

Case 3. $\widehat{\omega}\in \{0^n1^\infty: n\geq 1\}$.
Note that for any $n\geq 1$,
$$F_1\circ \underbrace{F_0\circ\cdots\circ F_0}_{n-\text{times}}(z)=F_1(z)>\frac{1}{2} \text{~ for all ~}z\in [\frac{1}{4},\frac{3}{8}).$$
Thus $\lim_{n\rightarrow\infty}\phi(n,x,\widehat{\omega})=1$.
\end{proof}

\begin{clm}\label{claim:2-3}
The point $\frac{3}{8}$ is not equi-invariant.
\end{clm}
\begin{proof}[Proof of Claim~\ref{claim:2-3}]
Suppose in contrast that $\frac{3}{8}$ is equi-invariant.
Then for any $\epsilon>0$ there exist $\delta>0$ and $\omega\in U^\mathbb{N}$ such that
\[\phi(\mathbb{N},y,\omega)\subset B(Q,\epsilon),\]
for all $y\in B(x,\delta)\cap Q$.
Since
$$U^\mathbb{N}=\{0^n1^\infty: n\geq 1\}\cup\{\omega: \omega_i\omega_{i+1}=10, \text{~for some~} i\geq 0\}\cup\{0^\infty, 1^\infty \},$$
by Claim~\ref{claim:2-2} and the proof of
Claim~\ref{claim:2-1}, we have $\omega=0^\infty$.
However, for any $z\in (\frac{3}{8}, \frac{1}{2}]$, we have $\lim_{n\rightarrow\infty}F_0^n(z)=1$;
this implies that $\lim_{n\rightarrow\infty}\phi(n,z,\omega)=1$.
This is a contradiction.
Therefore $\frac{3}{8}$ is not equi-invariant.
\end{proof}

By Claims \ref{claim:2-1} and ~\ref{claim:2-3},
the set $Q$ is finitely equi-invariant but not equi-invariant.

\end{proof}

\begin{rem}
By Claim~\ref{claim:2-3} and the proof of Claim~\ref{claim:2-1},
one can see that  every point in $Q$ is equi-invariant except the point $\frac{3}{8}$.
\end{rem}

The following example shows that  the equi-invariance
in the mean is strictly weaker than the equi-invariance.

\begin{exam}[EIM but not EI]\label{Exa:EIM-n-EI}
Consider a control system of form ~(\ref{eq:control-system-modle}), where
\begin{enumerate}
\item  $X=[0,1]=\{x\in\mathbb{R}:0\leq x \leq 1\}$;
\item  $U=\{0,1,2\}$;
\item  $F_0, F_1$ and $F_2: X\rightarrow X$ are defined by
\begin{equation*}
F_0(x)=
\begin{cases}
x, & \text {if $0\leq x<\frac{3}{8},$}\\
5(x-\frac{1}{2})+1, & \text {if $\frac{3}{8}\leq x<\frac{1}{2},$}\\
1, & \text {if $\frac{1}{2}\leq x< \frac{5}{8},$}\\
-5(x-\frac{3}{4})+\frac{3}{8}, & \text {if $\frac{5}{8}\leq x< \frac{3}{4},$}\\
\frac{3}{8}, & \text {if $\frac{3}{4}\leq x\leq 1,$}
\end{cases}
\end{equation*}
\begin{equation*}
F_1(x)=
\begin{cases}
1, & \text {if $0\leq x<\frac{1}{4},$}\\
-4(x-\frac{1}{4})+1, & \text {if $\frac{1}{4}\leq x<\frac{3}{8},$}\\
\frac{1}{2}, & \text {if $\frac{3}{8}\leq x<\frac{1}{2},$}\\
4(x-\frac{5}{8})+1, & \text {if $\frac{1}{2}\leq x<\frac{5}{8},$}\\
1, & \text {if $\frac{5}{8}\leq x\leq 1;$}
\end{cases}
\end{equation*}
and $F_2(x)=1$ for all $x\in [0,1].$
\end{enumerate}
Let $Q=[\frac{1}{4},\frac{1}{2}]$.
Then the set $Q$ is equi-invariant in the mean but not equi-invariant.
\end{exam}

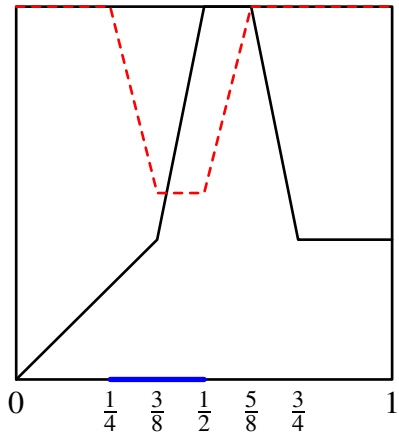
\begin{figure}[H]
  \centering
\begin{tikzpicture}[line width=1pt,line cap=round,line join=round,x=1.0cm,y=1.0cm,scale=0.618]
\clip(-1.5,-1.5) rectangle (9,9);
\draw (0,0)node[below]{$0$}--(2,0)node[below]{$\frac14$}--(3,0)node[below]{$\frac38$}
--(4,0)node[below]{$\frac12$}--(5,0)node[below]{$\frac58$}--(6,0)node[below]{$\frac34$}--(8,0)node[below]{$1$}--(8,8)--(0,8)--cycle;
\draw (0,0)--(3,3)--(4,8)--(5,8)--(6,3)--(8,3);
\draw[dashed,color=red] (0,8)--(2,8)--(3,4)--(4,4)--(5,8)--(8,8);
\draw[line width=2pt,color=blue] (2,0)--(4,0);
\end{tikzpicture}
  \caption{Equi-invariant in the mean.}\label{E-I-in-the-mean}
\end{figure}

\begin{proof}
Similar to the proof of Example~\ref{Exa:FEI-n-EI},
 almost all points in $Q$ are  equi-invariant except $\frac{3}{8}$.
So it suffices to show that $\frac{3}{8}$ is equi-invariant in the mean.
Indeed, for any $0<\epsilon<\frac{1}{8}$,
choose $N>0$ such that $\frac{1}{2N}<\varepsilon$ and $0<\delta'<\frac{1}{8}$ with
\begin{itemize}
\item $F^N_0(\frac{3}{8}+\delta')=\frac{1}{2}$ and
\item $F^i_0(\frac{3}{8}+\delta')<\frac{1}{2}$, $i=0,1,\ldots,N-1$.
\end{itemize}
Pick $\delta=\min\{\delta',\varepsilon\}$ and two control sequences $\omega=0^N20^\infty$,
then for any $y\in B(\frac{3}{8},\delta)$, there holds
\begin{itemize}
\item $F_{\omega_n}\circ\ldots\circ F_{\omega_0}(y)\in Q$ whenever $n=0,1,\ldots, N-1$;
\item $F_{\omega_N}\circ\ldots\circ F_{\omega_0}(y)\equiv 1$, this implies that
$$\frac1{N}\sum_{i=0}^{N}d[\phi(i,y,\omega),Q]
=\frac1{N}\sum_{i=0}^{N-1}d[\phi(i,y,\omega),Q]+\frac{d[\phi(N,y,\omega),Q]}{N}
=0+\frac{1}{2N}<\varepsilon;$$
\item $F_{\omega_n}\circ\ldots\circ F_{\omega_0}(y)\equiv\frac{3}{8}\in Q$ whenever $n>N$, this implies that
 \begin{align*}
 \frac1{n}\sum_{i=0}^{n}d[\phi(i,y,\omega),Q]
=&\frac1{n}\sum_{i=0}^{N-1}d[\phi(i,y,\omega),Q]+\frac{d[\phi(N,y,\omega),Q]}{N}+\frac1{n}\sum_{i=N+1}^{n}d[\phi(i,y,\omega),Q]\\
=&0+\frac{1}{2N}+0<\varepsilon,
 \end{align*}
 for all $n>N$.
\end{itemize}
Thus the point $\frac{3}{8}$ is equi-invariant in the mean.
\end{proof}

Before we give the forthcoming example, we recall some notions.
Let $I$ be a finite set. The one-sided symbolic space is
\[I^{\mathbb{N}_0} = \{x = (x_0, x_1,\ldots) : x_i \in I \text{~for~} i \in \mathbb{N}_0\}\]
with the distance
\begin{equation*}
\rho(x,y)=
\begin{cases}
0, & \text {if $x=y,$}\\
\frac1{i+1}, & \text {if $x\neq y$ and $i=\min\{j:x_j\neq y_j\}$}\\
\end{cases}
\end{equation*}
The \emph{shift} map $\sigma: I^{\mathbb{N}_0}\rightarrow I^{\mathbb{N}_0}$ is defined as
\[x = (x_0, x_1, \ldots)\mapsto \sigma(x) = (x_1, x_2, \ldots).\]
Then $(I^{\mathbb{N}_0}, \sigma)$ is a full shift.
For $\omega\in I^n$, the \emph{length} of $\omega$ is $l(\omega) = n$.
A \emph{cylinder} of $\omega$ is $[\omega] =\{x \in I^{\mathbb{N}_0}: (x_0,\ldots, x_{n-1}) = \omega\}$.

Now we give the following example to show that there exists a
 set which is  finitely equi-invariant in the mean but not finitely equi-invariant.

\begin{exam}[FEIM but not FEI]\label{Exa:EIM-n-FEI}
Let $I = \{a, b, c, d, e\}$,
$A = ab$, $B = cde$ and $\mathscr{B} = \{A, B\}$.
Consider a control system of form ~(\ref{eq:control-system-modle}), where
\begin{enumerate}
\item  $X = I^{\mathbb{N}_0}$;
\item  $U=\{0,1,2,3\}$;
\item  $ F_0$, $F_1,F_2$ and  $F_3: X\rightarrow X$ are defined by $F_0=\sigma^2$, $F_1=\sigma^3$,
$F_2\equiv b^\infty$ and
\begin{equation*}
F_3(x)=\begin{cases}
(ab)^\infty, & \text {if $x\in[b],$}\\
b^\infty, & \text {otherwise.}\\
\end{cases}
\end{equation*}
\end{enumerate}
Define an injective map
$\varphi$ from $\mathscr{B}^{\mathbb{N}_0}$ to $I^{\mathbb{N}_0}$ by
\[\varphi(\mu)_{[\sum_{j=0}^{i-1}l(\mu_i),\sum_{j=0}^il(\mu_i))}=\mu_i,\]
for $\mu\in \mathscr{B}^{\mathbb{N}_0}$.
Let $Q=\varphi(\mathscr{B}^{\mathbb{N}_0})$.
Then the set $Q$ is  equi-invariant in the mean but not finitely equi-invariant.
\end{exam}

\begin{proof}
By the construct of $Q$ and definitions of $F_0, F_1,F_2,F_3$, the set $Q$ is compact and
for any $x\in Q$ there exists an unique control sequence $\omega\in\mathscr{U}$ such that
$\phi(\mathbb{N}_0,x,\omega)\subset Q$. This implies that $Q$ is not finitely equi-invariant.

Next, we show that $Q$ is  equi-invariant in the mean.
Let $x\in Q$.
Fix any positive real number $\epsilon>0$.
Choose an integer with $n>\frac{1}{\epsilon}$.
Since the topology of the subspace $Q$ is
\begin{align*}
\mathscr{T}_Q=\{[u]:u\in I^n, n\geq 0\}\cap Q
=\{[(ab)^{n_1}(cde)^{m_1}&(ab)^{n_2}(cde)^{m_2}\cdots(ab)^{n_k}(cde)^{m_k}]:\\
&n_1+m_1+\cdots+n_k+m_k\geq 0,k\geq 1\},
\end{align*}
there exist $n_1,n_2,\ldots,n_k,m_1,m_2,\ldots,m_k$ such that
$$N:=n_1+m_1+\cdots+n_k+m_k>n,$$
 and
 $$x\in [(ab)^{n_1}(cde)^{m_1}(ab)^{n_2}(cde)^{m_2}\cdots(ab)^{n_k}(cde)^{m_k}].$$
Pick a control sequence $\omega=0^{n_1}1^{m_1}0^{n_2}1^{m_2}\cdots0^{n_k}1^{m_k}230^\infty$.
Then for any
$$y\in [(ab)^{n_1}(cde)^{m_1}(ab)^{n_2}(cde)^{m_2}\cdots(ab)^{n_k}(cde)^{m_k}],$$
we have
$$\phi([0,N],y,\omega)\subset Q, \phi(N+1,y,\omega)=b^\infty \text{~and~} \phi([N+2,\infty),y,\omega)=(ab)^\infty\in Q.$$
So
\begin{equation*}
\frac1{n}\sum_{i=0}^{n-1}d[\phi(i,y,\omega),Q]=\begin{cases}
0, & \text {if $n\leq N+1,$}\\
\frac1{n}d[\phi(N+1,y,\omega),Q]=\frac1{n}d(b^\infty,Q)\leq \frac1{N}<\epsilon, & \text {if $n\geq N+2.$}
\end{cases}
\end{equation*}
Thus, the point $x$ is equi-invariant in the mean and by the arbitrary of $x$,
 we have $Q$ is equi-invariant in the mean.
\end{proof}

Applying Theorem~\ref{thm:L-stable-equ-inv},
we provide an example which is  mean equi-invariant but not finitely equi-invariant in the mean.

\begin{exam}[MEI but not EIM; FMEI but not FEIM]\label{Exa:MEI-n-FEIM}
Consider a control system of form ~(\ref{eq:control-system-modle}), where
\begin{enumerate}
\item  $X=[0,1]=\{x\in\mathbb{R}:0\leq x \leq 1\}$;
\item  $U=\{0,1\}$;
\item  $F_0$ and $F_1: X\rightarrow X$ are defined by
\begin{equation*}
F_0(x)=
\begin{cases}
\frac{1}{2}, & \text {if $0\leq x<\frac{1}{4},$}\\
2(x-\frac{1}{2})+1, & \text {if $\frac{1}{4}\leq x<\frac{1}{2},$}\\
1, & \text {if $\frac{1}{2}\leq x\leq 1,$}\\
\end{cases}
\end{equation*}
and
\begin{equation*}
F_1(x)=
\begin{cases}
12(x-\frac{1}{4})^2+\frac{1}{4}, & \text {if $0\leq x<\frac{1}{4},$}\\
(x-\frac{1}{4})^2+\frac{1}{4}, & \text {if $\frac{1}{4}\leq x\leq 1.$}\\
\end{cases}
\end{equation*}
\end{enumerate}
Let $Q=[0,\frac{1}{4}]$.
Then the set $Q$ is  mean equi-invariant but not finitely equi-invariant in the mean.
\end{exam}

\begin{figure}[H]
  \centering
\begin{tikzpicture}[line width=1pt,line cap=round,line join=round,x=1.0cm,y=1.0cm,scale=0.618,declare function={
    f(\x)=1.5*(\x-2)^2+2;
    g(\x)=0.125*(\x-2)^2+2;
    }]
\clip(-1.5,-1.5) rectangle (9,9);
\draw (0,0)node[below]{$0$}--(2,0)node[below]{$\frac14$}
--(4,0)node[below]{$\frac12$}--(6,0)node[below]{$\frac34$}--(8,0)node[below]{$1$}--(8,8)--(0,8)--cycle;
\draw (0,4)--(2,4)--(4,8)--(8,8);
\draw[domain=0:2,color=red] plot (\x,{f(\x)});
\draw[domain=2:8,color=red] plot (\x,{g(\x)});
\draw[dashed] (0,0)--(8,8);
\draw[line width=2pt,color=blue] (0,0)--(2,0);
\end{tikzpicture}
  \caption{Mean equi-invariant}\label{M-E-I}
\end{figure}
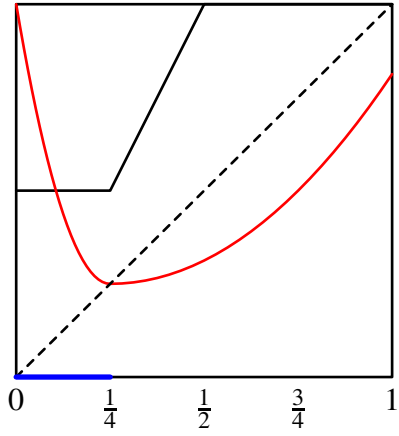

\begin{proof}
By definition of $F_1$, there hold
\begin{itemize}
\item[(a)]  $F_1(x)>F^2_1(x)>F^3_1(x)>\cdots>F^n_1(x)>\cdots$ and $\lim_{n\rightarrow\infty} F^n_1(x)=\frac{1}{4}$ for any $x\in [0,\frac{1}{4})$;
\item[(b)]  $F_1(x)>F_1(y)$ and $F_1^n(x)<F_1^n(y)$ for any $0\leq x<y\leq\frac{1}{4}$ and $n>1$.
\end{itemize}
Next, we divide our proof into two claims.
\begin{clm}\label{claim:2-17}
$Q$ is  mean equi-invariant.
\end{clm}
\begin{proof}[Proof of Claim~\ref{claim:2-17}]
Fix any $x\in Q$ and  $0<\epsilon<\frac{1}{4}$.
Since $\lim_{n\rightarrow\infty} F^n_1(0)=\frac{1}{4}$,
there exists $N'>0$ such that $d(F^n_1(x),Q)<\frac{\epsilon}{2}$ for all $n\geq N'$.
Choose a positive integer $N$ with $\frac{N'+1}{N}<\frac{\epsilon}{2}$, a control sequence $\omega=1^\infty$
and a positive real number $\delta=\epsilon$.
Then for all $n> N$,
\begin{align*}
\frac1{n}\sum_{i=0}^{n-1}d(\phi(i,y,\omega),Q)
&=\frac1{n}\sum_{i=0}^{N'}d(\phi(i,y,\omega),Q)+\frac1{n}\sum_{i=N'+1}^{n-1}d(\phi(i,y,\omega),Q)\\
&\leq\frac1{N}\sum_{i=0}^{N'}d(\phi(i,y,\omega),Q)+\frac1{n}\sum_{i=N'+1}^{n-1}d(\phi(i,y,\omega),Q)\\
&\leq\frac{N'+1}{N}+\frac{n-N'-1}{n}\cdot\frac{\epsilon}{2}\\
&<\frac{\epsilon}{2}+\frac{\epsilon}{2}=\epsilon,
\end{align*}
which implies that
\[\limsup_{n\rightarrow\infty}\frac1{n}\sum_{i=0}^{n-1}d(\phi(i,y,\omega),Q)\leq\epsilon.\]
Thus every point in $Q$ is finitely mean equi-invariant.
\end{proof}

\begin{clm}\label{claim:2-18}
$Q$ is not finitely  equi-invariant in the mean.
\end{clm}
\begin{proof}[Proof of Claim~\ref{claim:2-18}]
It suffices to show that $0$ is not finitely  equi-invariant in the mean;
that is there exists a positive real number $\epsilon>0$ such that for any
$\delta>0$, control sequences $\omega^{(1)},\ldots, \omega^{(k)}\in\mathscr{U}$ and $y\in B(0,\delta)\cap [0,\frac{1}{4}]$,  one can find some control sequence $\omega^{(r)}$, $1\leq r\leq k$,
satisfies that
\[\frac1{n}\sum_{i=0}^{n-1}d[\phi(i,y,\omega^{(r)}),Q]\geq\varepsilon,\]
for some $n$.
Indeed, for any $0<\delta<\frac{\sqrt{3}-1}{4\sqrt{3}}$, control sequence
$\omega\in\mathscr{U}$ and $y\in B(0,\delta)\cap [0,\frac{1}{4}]$, there holds
\begin{align*}
&\frac{1}{2}\Large(d[\phi(0,y,\omega),Q]+d[\phi(1,y,\omega),Q]\Large)\\
\geq& \frac{1}{2}d[\phi(1,y,\omega),Q]\\
=&\frac{1}{2}d[F_0(y),Q]\\
=&\frac{1}{2}
\end{align*}
whenever $\omega_0=0$;
and
\begin{align*}
&\frac{1}{2}\Large(d[\phi(0,y,\omega),Q]+d[\phi(1,y,\omega),Q]\Large)\\
\geq& \frac{1}{2}d[\phi(1,y,\omega),Q]\\
=&\frac{1}{2}d[F_1(y),Q]\\
\geq&\frac{1}{2}
\end{align*}
whenever $\omega_0=1$.
\end{proof}
\end{proof}

Next, we provide an example which is finitely mean equi-invariant but not mean equi-invariant.

\begin{exam}[FMEI but not MEI]\label{Exa:FMEI-n-MEI}
Consider a control system of form ~(\ref{eq:control-system-modle}), where
\begin{enumerate}
\item  $X=[0,1]=\{x\in\mathbb{R}:0\leq x \leq 1\}$;
\item  $U=\{0,1\}$;
\item  $ F_0$, and $F_1: X\rightarrow X$ are defined by
\begin{equation*}
F_0(x)=
\begin{cases}
\frac{1}{8}, & \text {if $0\leq x<\frac{1}{4},$}\\
2x-\frac38, & \text {if $\frac{1}{4}\leq x<\frac{3}{8},$}\\
(x-\frac38)^2+\frac38, & \text {if $\frac{3}{8}\leq x\leq1,$}\\
\end{cases}
\end{equation*}
and
\begin{equation*}
F_1(x)=
\begin{cases}
0, & \text {if $0\leq x<\frac{1}{16},$}\\
2x-\frac18, & \text {if $\frac{1}{16}\leq x<\frac{1}{8},$}\\
\frac14(x-\frac18)^{\frac{1}{3}}+\frac18, & \text {if $\frac{1}{8}\leq x<\frac14,$}\\
-x+\frac1{2}, & \text {if $\frac{1}{4}\leq x<\frac12,$}\\
0, & \text {if $\frac12\leq x\leq1.$}\\
\end{cases}
\end{equation*}
\end{enumerate}
Let $Q=[\frac14,\frac12]$.
Then the set $Q$ is finitely mean equi-invariant but not mean equi-invariant.
\end{exam}

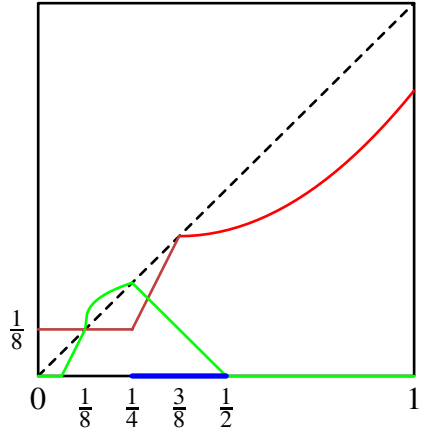
\begin{figure}[H]
  \centering
\begin{tikzpicture}[line width=1pt,line cap=round,line join=round,x=1.0cm,y=1.0cm,scale=0.618,declare function={
    f(\x)=(\x-3)^2/8+3;
    g(\x)=(\x-1)^(1/3)+1;
    }]
\clip(-1.5,-1.5) rectangle (9,9);
\draw (0,0)node[below]{$0$}--(1,0)node[below]{$\frac18$}
--(2,0)node[below]{$\frac14$}--(3,0)node[below]{$\frac38$}--
(4,0)node[below]{$\frac12$}--(8,0)node[below]{$1$}--(8,8)--(0,8)--(0,1)node[left]{$\frac18$}--cycle;
\draw[dashed] (0,0)--(8,8);
\draw[domain=3:8,color=red] plot (\x,{f(\x)});
\draw[domain=1:2,color=green] plot (\x,{g(\x)});
\draw[gray!50!red] (2,1)--(3,3) (0,1)--(2,1);
\draw[green] (2,2)--(4,0)--(8,0) (1,1)--(0.5,0)--(0,0);

\draw[line width=2pt,color=blue] (2,0)--(4,0);

\end{tikzpicture}
  \caption{Finitely mean equi-invariant}\label{F-M-E-I}
\end{figure}

\begin{proof}
By definitions of $F_0$ and $F_1$, we have the following properties:
\begin{itemize}
\item[(a)]  $F_0(x)\geq F^2_0(x)\geq F^3_0(x)\geq\cdots\geq F^n_0(x)\geq\cdots$,  for any $x\in [\frac{1}{4},\frac{3}{8})\cup (\frac{3}{8},\frac{1}{2}] $;
\item[(b)]   $\lim_{n\rightarrow\infty} F^n_0(x)=\frac{1}{8}$
for any $x\in [\frac{1}{4},\frac{3}{8})$ and $\lim_{n\rightarrow\infty} F^n_0(x)=\frac{3}{8}$ for any $x\in (\frac{3}{8},\frac{1}{2}] $;
\item[(c)]  $F_0^n(x)\leq F_0^n(y)$ for any $\frac{1}{4}\leq x\leq y\leq\frac{1}{2}$ and $n\geq 0$,
$F_1^n(x)\leq F_1^n(y)$ for any $\frac{1}{8}\leq x\leq y\leq\frac{1}{4}$ and $n\geq 0$;
\item[(d)] $F_1(x)< F^2_1(x)<F^3_1(x)<\cdots<F^n_1(x)<\cdots$ and $\lim_{n\rightarrow\infty} F^n_1(x)=\frac{1}{4}$ for any $x\in (\frac{1}{8},\frac{1}{4})$, and $F_1(\frac{1}{4})=\frac{1}{4}$;
 \item[(e)]   $F_1(x)\geq F^2_1(x)\geq F^3_1(x)\geq\cdots\geq F^n_1(x)\geq\cdots$ and  $\lim_{n\rightarrow\infty} F^n_1(x)=0$ for any $x\in (\frac{3}{8},\frac{1}{2}] $;
 \item[(f)]   $F_0(x),F_1(x)\in [0,\frac{1}{8}]$ for all $x\in[0,\frac1{8}]$.
\end{itemize}

Next, we divide our proof into three claims.

\begin{clm}\label{claim:3-1}
 Every point in $Q\backslash \{\frac{3}{8}\}$ is  mean equi-invariant.
\end{clm}
\begin{proof}[Proof of Claim~\ref{claim:3-1}]
Case 1.  $x\in (\frac{3}{8},\frac{1}{2}]$.
Take $\omega=0^\infty$. Then,
 by Properties (a), (b) and (c) above,
we have
$d[\phi(n,x,\omega),Q]=0$
for all $n\geq 0$.
It follows that $x$ is mean equi-invariant.

Case 2.  $x\in [\frac{1}{4},\frac{3}{8})$.
Then $\frac{1}{8}<F_1(x)\leq\frac{1}{4}$.
Take $\omega=1^\infty$. Then for any $\epsilon>0$,
 by Property (d) above,
we have
\[\limsup_{n\rightarrow\infty}\frac1{n}\sum_{i=0}^{n-1}d[\phi(i,x,\omega),Q]<\epsilon.\]
According to Property (c) above, we have $x$ is mean equi-invariant.
\end{proof}

\begin{clm}\label{claim:3-2}
The point  $\frac{3}{8}$ is  finitely mean equi-invariant.
\end{clm}
\begin{proof}[Proof of Claim~\ref{claim:3-2}]
It comes directly from the proof of Claim~\ref{claim:3-1}.
\end{proof}

\begin{clm}\label{claim:3-3}
The point  $\frac{3}{8}$ is  not mean equi-invariant.
\end{clm}
\begin{proof}[Proof of Claim~\ref{claim:3-3}]
Suppose in contrast that $\frac{3}{8}$ is  mean equi-invariant;
that is for any $\epsilon>0$ there exist $\delta>0$ and $\omega\in\mathscr{U}$ such that
\[\limsup_{n\rightarrow\infty}\frac1{n}\sum_{i=0}^{n-1}d[\phi(i,y,\omega),Q]<\epsilon\]
for all $y\in(\frac3{8}-\delta,\frac3{8}+\delta)$.
Indeed, on one hand, for $x\in (\frac{3}{8},\frac{1}{2}]$,
let $\omega\in\mathscr{U}$.
If $\omega\in\{1\omega':\omega'\in\mathscr{U}\}$,
then $\phi(1,x,\omega)=F_1(x)\in[0,\frac1{8}]$ and $\phi(n,x,\omega)\in[0,\frac1{8}]$ for all $n>1$ by Property (f) above.
This implies that
\[\limsup_{n\rightarrow\infty}\frac1{n}\sum_{i=0}^{n-1}d[\phi(i,y,\omega),Q]\geq\frac1{8}.\]
If $\omega=0^n1\omega'$ for some $n\geq 1$ and $\omega'\in\mathscr{U}$,
then $\phi(n+1,x,\omega)=F_1(x)\in[0,\frac1{8}]$ and consequently
\[\limsup_{n\rightarrow\infty}\frac1{n}\sum_{i=0}^{n-1}d[\phi(i,y,\omega),Q]\geq\frac1{8}.\]
Since
\[\mathscr{U}=\{0^\infty\}\cup\{1\omega':\omega'\in\mathscr{U}\}\cup\{0^n1\omega':n\geq 1\text{~and~}\omega'\in\mathscr{U}\},\]
by the proof of Claim~\ref{claim:3-1},
 there exists only one control sequence $\omega=0^\infty$ such that $$\limsup_{n\rightarrow\infty}\frac1{n}\sum_{i=0}^{n-1}d[\phi(i,y,\omega),Q]<\frac1{8}.$$
On the other hand, if $x\in [\frac{1}{4},\frac{3}{8})$.
 Take $\omega=0^\infty$.
 By Property (b), there exists $N>0$ such that
 $$d[\phi(n,y,\omega),\frac1{4}]=d[\phi(n,y,\omega),Q]\geq\frac1{16}$$ for all $n\geq N$.
 Thus
 \begin{align*}
\limsup_{n\rightarrow\infty,n>N}\frac1{n}\sum_{i=0}^{n-1}d[\phi(i,x,\omega),Q]
=&\limsup_{n\rightarrow\infty,n>N}\frac1{n}\LARGE(\sum_{i=0}^{N-1}d[\phi(i,x,\omega),Q]+\sum_{i=N}^{n-1}d[\phi(i,x,\omega),Q]\LARGE)\\
\geq&\limsup_{n\rightarrow\infty,n>N}\frac1{n}\sum_{i=N}^{n-1}d[\phi(i,\phi(m,x,\omega),\omega),Q])\\
\geq&\frac1{16}.
\end{align*}

This is a contradiction. So $\frac{3}{8}$ is  not mean equi-invariant.
\end{proof}
By Claims~\ref{claim:3-1} and \ref{claim:3-2}, the set $Q$ is finitely mean equi-invariant.
By Claim~\ref{claim:3-3}, it is not mean equi-invariant.
\end{proof}


\section*{Acknowledgements}
The authors thank Professor Jian Li for sharing his research on dynamical systems.
The first and third authors were supported by National Nature Science Funds of China (11771459) and the first author was also supported by Research Funds of Guangdong University of Foreign Studies(299-X5218165 and 299-X5219222);
the second author was supported by the National Natural Science Foundation of China
(Nos.11701584 and 11871228) and the Natural Science Research Project of Guangdong Province (Grant No.2018KTSCX122).

\end{document}